\documentclass[11pt, a4paper]{amsart}
\usepackage{amssymb}
\usepackage[margin=30mm]{geometry}
\newcommand{\PSp}{\mathop{\mathrm{PSp}}}
\newcommand{\soc}{\mathop{\mathrm{soc}}}

\newcommand{\PGL}{\mathop{\mathrm{PGL}}}
\newcommand{\PSU}{\mathop{\mathrm{PSU}}}

\newcommand{\PSL}{\mathrm{\mathrm{PSL}}}

\newcommand{\GL}{\mathop{\mathrm{GL}}}

\newcommand{\lcm}{\mathop{\mathrm{lcm}}}
\newcommand{\PGammaL}{\mathop{\mathrm{P}\Gamma\mathrm{L}}}

\newcommand{\Alt}{\mathop{\mathrm{Alt}}}
\newcommand{\Sym}{\mathop{\mathrm{Sym}}}

\renewcommand{\wr}{\mathop{\mathrm{wr}}}
\newtheorem{theorem}{Theorem}[section]
\def\cent#1#2{{\bf C}_{#1}(#2)}

\def\PGdq{{\rm PG}_{d-1}(q)}

\newtheorem{proposition}[theorem]{Proposition}
\newtheorem{lemma}[theorem]{Lemma}
\newenvironment{theorem*}[1][Theorem.]{\begin{trivlist}\item[\hskip \labelsep {\bfseries #1}]}{\end{trivlist}}
\newtheorem{remark}[theorem]{Remark}
\newtheorem{notation}[theorem]{Notation}
\begin{document}

\title[Four cycles]{Finite primitive permutation groups containing a permutation having at most four cycles} 

\author[S.~Guest]{Simon Guest}
\address{Simon Guest, Department of Mathematics, University of Southampton, \newline 
Highfield, SO17 1BJ, United Kingdom}\email{s.d.guest@soton.ac.uk}

\author[J. Morris]{Joy Morris}
\address{Joy Morris, Department of Mathematics and Computer Science,
University of Lethbridge,\newline
Lethbridge, AB. T1K 3M4. Canada}
\email{joy@cs.uleth.ca}

\author[C. E. Praeger]{Cheryl E. Praeger}
\address{Cheryl E. Praeger, Centre for Mathematics of Symmetry and Computation,
School of Mathematics and Statistics,
The University of Western Australia,\newline
 Crawley, WA 6009, Australia}\email{cheryl.praeger@uwa.edu.au}

\author[P. Spiga]{Pablo Spiga}
\address{Pablo Spiga,
Dipartimento di Matematica e Applicazioni, University of Milano-Bicocca,\newline
Via Cozzi 53, 20125 Milano, Italy}\email{pablo.spiga@unimib.it}

\thanks{Address correspondence to P. Spiga,
E-mail: pablo.spiga@uwa.edu.au\\
The second author is supported in part by the National Science
 and Engineering Research Council of Canada.
 The third author is supported by the ARC Federation
Fellowship Project
FF0776186. The fourth author is supported by UWA as part of the ARC Federation Fellowship Project. }


\begin{abstract}
We classify the finite primitive groups containing a permutation with at most four cycles (including fixed points) in its disjoint cycle representation. 
\end{abstract}
\maketitle

\section{Introduction}\label{introduction}

In this paper, we are interested in permutations that have a small number of cycles (including fixed points) in their disjoint cycle representations.
Primitive permutation groups containing elements with few cycles have been studied for over a century, because of their importance in various areas of mathematics. In particular, interest in primitive groups on $n$ points containing an $n$-cycle dates back to a 1911 theorem by Burnside~\cite[1 \S 251--252]{burnside}, and a complete classification of such groups, based on the Finite Simple Group Classification, was achieved in 1997~\cite{Feit,jones,mcsorley}. This classification has important consequences, including the investigation of cyclic codes, cyclic designs, Cayley graphs on cyclic groups, and rotary embeddings of graphs in surfaces.
Motivated by certain number-theoretic applications which we mention briefly in Section 1.2, M\"uller~\cite{mueller} classified primitive groups containing permutations with two cycles. 

The aim of this paper is to classify the primitive groups containing a permutation with at most four cycles in its disjoint cycle representation, including fixed points.  Our classification will be applied to a problem about normal coverings of symmetric and alternating groups, and we give some details about this application in Section 1.2.

Such primitive groups turn out to be quite rare, and we are able to classify them according to the permutation actions of their socles. We are further able to list the possible $N$-tuples of cycle lengths that can appear in a group with each possible socle, where $N \le 4$. Some of the techniques used are similar to those used by M\"uller in~\cite{mueller}.

\subsection{Main theorem and proof strategy.}

Our main result is the following.
\begin{theorem}\label{MAIN}
Let $G$ be a finite primitive permutation group of degree $n$, and let $g$ be an element of $G$ having cycle lengths (counted with multiplicity) $(n_1,\ldots,n_N)$ with $N\leq 4$.  Then the socle $\soc(G)=T^\ell$ for some simple group $T$ and integer $\ell$, and
 one of the following holds:
\begin{enumerate}
\item[(i)] $\soc(G)=\Alt(m)^\ell$ in its natural product action of degree $m^\ell$, with $\ell\leq 3$;
\item[(ii)] $\soc(G)=\PSL_d(q)^\ell$ in its natural product action of degree $((q^d-1)/(q-1))^\ell$, with $\ell\leq 3$ and $d \ge 2$;
\item[(iii)] $T$, $\ell$, $n$ and $(n_1,\ldots,n_N)$ are in one of the rows of Tables~\ref{BIGtablea},~\ref{BIGtableb} or~\ref{BIGtablec};
\item[(iv)] $\soc(G)$ is elementary abelian, and $g$ and $G$ are described in~Theorems~1.3 and~1.4 of \cite{GMPSaffine}.
\end{enumerate}
Moreover, Tables~\ref{table1} and~\ref{table2} list all of the possible $N$-tuples $(n_1,\ldots,n_N)$ arising from parts~$(i)$ and~$(ii)$ respectively.
\end{theorem}

The cycle lengths listed in Tables~\ref{table1} and~\ref{table2} may not be possible in every group that has the given socle, but do occur in at least some such groups. For example, in $\Alt(m)$ the possible cycle lengths depend upon the parity of $m$, but $\Alt(m)$ is the socle of $\Sym(m)$ and in $\Sym(m)$ any cycle lengths are possible, so this is what we have listed.

It is an elementary exercise to show that for a permutation group element $g$, the order of $g$ is the least common multiple of the lengths of the cycles into which it can be decomposed. If $g$ has at most four cycles and lies inside a permutation group of degree $n$, the pigeon-hole principle guarantees that one of those cycles must have length at least $n/4$; so clearly $|g|\ge \lceil n/4\rceil$. It is therefore natural that our proof of Theorem~\ref{MAIN} uses the recent classification of the finite primitive groups $G$ of degree $n$ containing a permutation $g$ with the order $|g|$ of $g$ at least $n/4$, see~\cite[Theorem~$1.3$]{GMPSorders}. In particular, in order to prove Theorem~\ref{MAIN} it suffices to investigate which groups in~\cite[Theorem~$1.3$]{GMPSorders} actually contain a permutation with at most four cycles. Comparing Theorem~\ref{MAIN} with~\cite[Theorem~$1.3$]{GMPSorders} we see that the number of examples is actually very limited. For the benefit of the reader we record the statement of~\cite[Theorem~$1.3$]{GMPSorders}.

\begin{theorem*}{{(Theorem~$1.3$~\cite{GMPSorders})}}
\emph{
Let $G$ be a finite primitive group of degree $n$ and assume that $G$ contains a
permutation $g$ with $|g|\geq n/4$. Then the socle $\soc(G)=T^\ell$ of $G$ is isomorphic to either
\begin{enumerate}
\item[(i)] $\Alt(m)^\ell$ in its natural product action on $\ell$ direct copies of the $k$-subsets of
$\{1,\ldots,m\}$, or to
\item[(ii)] $\PSL_d(q)^\ell$ in its natural action on $\ell$ direct copies of the points or hyperplanes of the projective space $\PGdq$, or to
\item[(iii)] an elementary abelian group $C_p^{\ell}$, and $G$ and $g$ are described in~\cite{GMPSaffine}, or to
\item[(iv)] one of the groups in~\cite[Table~2]{GMPSorders}.
\end{enumerate}
Moreover, there exists a positive integer $\ell_T$ depending only on $T$ with $\ell\leq \ell_T$.}
\end{theorem*}
For each group $G$ in~\cite[Table~2]{GMPSorders}, the exact value of $\ell_T$ is given in~\cite[Table~6]{GMPSorders}. In particular, for each of the groups arising from~\cite[Theorem~$1.3$~(iv)]{GMPSorders}, the proof of Theorem~\ref{MAIN} follows with an immediate computation in \texttt{magma}~\cite{magma}: for each such group we have given in Tables~\ref{BIGtablea},~\ref{BIGtableb} and~\ref{BIGtablec} all the examples admitting a permutation with at most four cycles. Therefore, for the rest of this paper, we may assume that the permutation group $G$ is as in part~(i) or~(ii) of~\cite[Theorem~$1.3$]{GMPSorders}.

\subsection{Motivation.} 
Classifications of this type are of interest from a computational and from a theoretical point of view. For instance, finding the order or the number of cycles of a permutation is a very inexpensive operation that a computer can perform as part of a recognition algorithm. In particular, an 1863 theorem of Jordan underpins a Monte Carlo algorithm that is used by major computer algebra systems, {\sf GAP}~\cite{GAP4} and \texttt{magma}~\cite{magma}, to determine whether or not a given set of permutations of $n$ points generates $\Alt(n)$ or $\Sym(n)$. Jordan's theorem guarantees that the only primitive groups of degree $n$ containing an element with one cycle of prime length $p$, and $n-p \ge 3$ fixed points, are $\Alt(n)$ or $\Sym(n)$. 
Other theoretical results have also been important in various computer tests. For example the studies in~\cite{LM,PrPr} on the primitive groups containing a permutation of prime order $p$ and having at most $p-1$ cycles of length $p$ have been used to improve the Monte Carlo algorithm described above.

Theoretical applications using lists of primitive groups containing permutations of a certain structure are numerous. We describe briefly a few of these applications here. The application to the covering number of $\Sym(m)$ is our main motivation.

%First, the primitive groups that contain a regular cyclic subgroup are pivotal in the study of regular embeddings of graphs in surfaces and in the study of the dessin d'enfant REFERENCE HERE.

Given a finite group $G$, the  normal coverings of $G$ are the families $H_1,\ldots,H_r$ of proper subgroups of $G$ such that each element of $G$ has a conjugate in $H_i$ for some $i\in \{1,\ldots,r\}$. The minimum $r$ is usually denoted by $\gamma(G)$. For $G=\Alt(m)$ or $\Sym(m)$, it is shown in~\cite{BP} that $a\varphi(m)\leq \gamma(G)\leq bm$ for positive constants $a$ and $b$ (where $\varphi$ denotes Euler's totient function). More recently, Bubboloni, Praeger and Spiga~\cite{BPS} developed some new research on this topic starting with the idea that primitive subgroups of the symmetric group are ``few and small'' and therefore cannot play a significant role in normal coverings. In particular, as an application of Theorem~\ref{MAIN}, they considerably improve the lower bound on $\gamma(G)$.
The normal coverings of the symmetric and alternating groups also play a role in Galois theory. For example, let $f(x)\in\mathbb{Z}[x]$ be a polynomial that has a root modulo $p$ for all primes $p$ and yet has no root in $\mathbb{Q}$. Consider the Galois group $G$ of $f(x)$ over $\mathbb{Q}$. Remarkably, the number of irreducible factors of $f(x)$ over $\mathbb{Q}$ is at least $\gamma(G)$ by~\cite[Theorem~2]{42}. 

Finite primitive groups containing a permutation with few cycles also play a crucial role in the study of the monodromy groups of Siegel functions and in the proof of a stronger version of Hilbert's irreducibility theorem~\cite{mueller}. Here we give a brief account of the relationship between these ostensibly unrelated topics and we refer the interested reader to the introduction of~\cite{mueller} for details. Let $k$ be a number field, let $\mathcal{O}_k$ denote the ring of integers of $k$, and let $f(t,X) \in k(t)[X]$ be an irreducible polynomial (as usual $k(t)$ denotes the field of $k$-rational functions in the variable $t$). Hilbert's irreducibility theorem states that $f(t_0,X)$ is irreducible over $k$ for
infinitely many integral specialisations $t_0 \in \mathcal{O}_k$. However, there is no control whatsoever on which values of $t_0$ can be used. Using a remarkable theorem of Siegel, one can see that in order to obtain
a refined version of Hilbert's irreducibility theorem, one has to obtain information on the rational functions $g(X)$ such that $\{g(a)\mid a \in k\}\cap \mathcal{O}_k$ has infinite cardinality. Such functions are now called Siegel functions. Another theorem of Siegel shows that a Siegel function $g(X)$ has at most two poles on the Riemann sphere. Thus, the Galois group of $g(X)-t$ over $k(t)$ contains a permutation with at most two cycles. Therefore the list of the finite primitive groups containing a permutation with at most two cycles is essential for the refinement of Hilbert's irreducibility theorem obtained in~\cite{mueller}.

There is another number theoretic application that we wish to discuss briefly and we thank  M\"{u}ller for bringing this to our attention (we refer the reader to the beautiful introduction in~\cite{gems} for more details). Given a positive integer $\ell$, a {\em lacunary} rational function (with respect to $\ell$) is an expression $f(t)=P(t)/Q(t)$, with $P(t)$ and $Q(t)$ polynomials having at most $\ell$ summands in total. The main result of~\cite{gems} describes the decompositions $f(t)=g(h(t))$ with $g,h\in k(t)$ of degree at least $2$. In this deep investigation, again the list of the finite primitive groups admitting a permutation with at most two cycles is absolutely essential.

We consider it very likely that there are other applications of such results, and we hope that the existence of this result will prove useful and motivational to researchers in a variety of fields.

\subsection{Structure of the paper}\label{struc}In Section~\ref{sec2}, we have a closer look at the action of $\Sym(m)$ on $k$-sets and we show that (for $k\geq 2$) such a group contains a permutation with at most four cycles only when $k\leq 3$ and $m\leq 9$. In Section~\ref{sec3} we study $\PGammaL_d(q)$ in its natural action on the points or hyperplanes of the projective space $\PGdq$ and  determine the cycle structure of the permutations that have at most four cycles. We use the results of Sections~\ref{sec2} and~\ref{sec3} in Section~\ref{notation}, to prove Theorem~\ref{MAIN}. All of our tables are contained in Section~\ref{sec:tables}.

\section{The Tables}\label{sec:tables}

\begin{table}[!ht]
\begin{tabular}{|cccl|}\hline
Line&$\ell$&Cycle Lengths&Remarks\\\hline\hline
1&$1$ &$k_1,k_2,k_3,k_4$&$k_1+k_2+k_3+k_4=m$\\\hline
2&$2$ &$mk$, $m(m-k)$&$\gcd(m,k)=1$\\
3&$2$ &$mk_1$, $mk_2$, $mk_3$&$m=k_1+k_2+k_3$ and\\
 & & &$\gcd(m,k_i)=1$ for $i=1,2,3$\\
4&$2$ &$k_1k_2$, $k_1(m-k_2)$, $k_2(m-k_1)$, & $k_1$ and $m-k_1$ coprime\\
 & &$(m-k_1)(m-k_2)$ &to $k_2$ and $m-k_2$\\

5&$2$ &$\frac{k_1m}{2}$, $\frac{k_1m}{2}$, $k_2m$, $k_3m$&$m=k_1+k_2+k_3$,\\
 & & & $\gcd(m,k_1)=2$,\\
 & & &$\gcd(m,k_i)=1$ for $i=2,3$\\
6&$2$ &$k_1m$, $k_2m$, $k_3m$, $k_4m$&$m=k_1+k_2+k_3+k_4$\\
 & & &$\gcd(m,k_i)=1$ for $i=1,2,3,4$\\
7&$2$ &$3$, $4$, $6$, $12$ or $1$, $8$, $8$, $8$&$m=5$\\
8&$2$ & $12$, $12$, $12$&$m=6$\\
9&$2$ & $16$, $16$, $16$, $16$&$m=8$\\\hline
10&$3$ &$mk_1k_2$, $mk_1(m-k_2)$, $mk_2(m-k_1)$,&$m,k_1,m-k_1,k_2$ and $m-k_2$\\
 & & $m(m-k_1)(m-k_2)$& pairwise coprime\\
11&$3$ &$5$, $40$, $40$, $40$ or $15$, $20$, $30$, $60$& $m=5$\\\hline
\end{tabular}\caption{Cycle lengths $(n_1,\ldots,n_N)$ for Theorem~\ref{MAIN}~(i)}\label{table1}
\end{table}

\begin{table}[!ht]
\begin{center}
\begin{tabular}{|cccll|}\hline
Line&$\ell$&$d$ & $q$ & Cycle Lengths \\ \hline\hline
$1$ & $1$& any & any &$\frac{q^d-1}{q-1}$ \\
$2$ & $1$ & even &odd &$\frac{1}{2}\left(\frac{q^d-1}{q-1}\right),\frac{1}{2}\left(\frac{q^d-1}{q-1}\right)$ \\
$3$ & $1$ &even &$\equiv 2\mod{3}$ &$\frac{1}{3}\left(\frac{q^d-1}{q-1}\right),\frac{1}{3}\left(\frac{q^d-1}{q-1}\right),\frac{1}{3}\left(\frac{q^d-1}{q-1}\right)$ \\
$4$ &$1$ &$3 \mid d$ &$\equiv 1\mod{3}$ &$\frac{1}{3}\left(\frac{q^d-1}{q-1}\right),\frac{1}{3}\left(\frac{q^d-1}{q-1}\right),\frac{1}{3}\left(\frac{q^d-1}{q-1}\right)$ \\
$5$ & $1$ &$4 \mid d$ &$\equiv 1\mod{4}$ &$\frac{1}{4}\left(\frac{q^d-1}{q-1}\right),\frac{1}{4}\left(\frac{q^d-1}{q-1}\right),\frac{1}{4}\left(\frac{q^d-1}{q-1}\right),\frac{1}{4}\left(\frac{q^d-1}{q-1}\right)$ \\
$6$ & $1$ &even &$\equiv 3\mod{4}$ &$\frac{1}{4}\left(\frac{q^d-1}{q-1}\right),\frac{1}{4}\left(\frac{q^d-1}{q-1}\right),\frac{1}{4}\left(\frac{q^d-1}{q-1}\right),\frac{1}{4}\left(\frac{q^d-1}{q-1}\right)$ \\
$7$ & $1$ &$d=d_1+d_2$ &any &$\frac{q^{d_1}-1}{q-1},\frac{q^{d_2}-1}{q-1}, \frac{(q^{d_1}-1)(q^{d_2}-1)}{q-1}$ \\
 & &$\gcd(d_1,d_2)=1$ & & \\
 & &$qd_1d_2$ even & & \\
$8$ & $1$ & $d=d_1+d_2$ & odd & $\frac{q^{d_1}-1}{q-1},\frac{q^{d_2}-1}{q-1}, \frac{(q^{d_1}-1)(q^{d_2}-1)}{2(q-1)}$, $\frac{(q^{d_1}-1)(q^{d_2}-1)}{2(q-1)}$\\
 & &$\gcd(d_1,d_2)=1$ & &\\
 & &$d_1,d_2$ odd & & \\
$9$ & $1$ &$2$ &prime &$1,q$ \\
$10$ & $1$ &$2$ &4 &$1,2,2$ \\
$11$& $1$ &$2$ &$8$ &$1,2,6$\\
$12$& $1$ &$2$ &$9$ &$1,3,3,3$, or $1,3,6$, or $2,4,4$, or $2,8$\\
$13$& $1$ &$2$&$16$& $1,8,8$, or $2,3,12$, or $2,5,10$\\
$14$& $1$ &$2$&$25$& $2,12,12$, or $1,5,10,10$\\
$15$& $1$ &$2$&$27$& $4,12,12$, or $1,9,9,9$\\
$16$& $1$ &$2$&$32$& $3,15,15$\\
$17$& $1$ &$2$&$49$& $2,16,16,16$\\
$18$& $1$ &$3$&$2$& $1,2,4$ \\
$19$& $1$ &$3$&$4$&$1,4,8,8$, or $7,14$\\
$20$& $1$ &$3$&$9$&$13,26,26,26$\\
$21$& $1$ &$4$ &$2$ &$3,6,6$ \\
$22$& $1$ &$4$ &$3$ &$4,12,12,12$ \\
$23$& $1$ &$4$ &$4$ &$10,15,30,30$\\\hline

$24$&$2$ &$d=d_1+d_2$ &$\gcd(q-1,d)=1$&$\frac{(q^d-1)(q^{d_1}-1)}{(q-1)^2}$, $\frac{(q^d-1)(q^{d_2}-1)}{(q-1)^2}$,\\
 & &$\gcd(d_1,d_2)=1$& &$\frac{(q^d-1)(q^{d_1}-1)(q^{d_2}-1)}{(q-1)^2}$\\

$25$&$2$ &$d=d_1+d_2$ &$\gcd(q-1,d)=2$&$\frac{(q^d-1)(q^{d_1}-1)}{(q-1)^2}$, $\frac{(q^d-1)(q^{d_2}-1)}{(q-1)^2}$,\\
 & &$\gcd(d_1,d_2)=1$& &$\frac{(q^d-1)(q^{d_1}-1)(q^{d_2}-1)}{2(q-1)^2}$, $\frac{(q^d-1)(q^{d_1}-1)(q^{d_2}-1)}{2(q-1)^2}$\\
$26$ & $2$&$2$&prime&$q+1,q(q+1)$\\
$27$ & $2$&$2$&prime&$\frac{q+1}{2},\frac{q+1}{2},\frac{q(q+1)}{2},\frac{q(q+1)}{2}$\\
$28$ & $2$ &$2$&$4$&$5,10,10$\\
$29$ & $2$ &$2$&$5$&$12,12,12$\\
$30$ & $2$ &$2$&$7$&$16,16,16,16$\\
$31$ & $2$ &$2$&$9$&$10,30,30,30$\\
$32$ & $2$ &$2$&$16$&$17,136,136$, or $34,51,204$, or $34,85,170$\\
$33$ & $2$ &$2$&$27$&$28,252,252,252$\\
$34$ & $2$ &$3$&$2$&$7,14,28$, or $7,14,14,14$\\
$35$ & $2$ &$3$&$4$&$21,84,168,168$\\\hline
\end{tabular}\caption{Cycle lengths $(n_1,\ldots,n_N)$ for Theorem~\ref{MAIN}~(ii)}\label{table2}
\end{center}
\end{table}

\begin{table}[!ht]
\begin{center}
\begin{tabular}{|ccll|}\hline
Socle Factor&$\ell$&degree& Cycle Lengths \\ \hline\hline
$\Alt(5)$&$1$ &$10$& $(5, 5)$, $(2, 4, 4)$, $(1, 3, 6)$, $(1, 3, 3, 3)$\\
 &$2$& $60$&$(15, 15, 15, 15)$\\\hline
$\Alt(6)$&$1$& $15$&$(5,5,5)$, $(3,6,6)$\\\hline
$\Alt(7)$&$1$&$15$ &$(5,5,5)$, $(3,6,6)$, $(1,7,7)$\\
 &$1$&$21$ &$(7,7,7)$, $(1,5,5,10)$, $(2,3,4,12)$, $(3,6,6,6)$\\\hline
$\Alt(8)$&$1$&$28$ &$(7,7,7,7)$, $(4,8,8,8)$, $(3,5,5,15)$\\
 &$1$&$35$ &$(5,15,15)$\\\hline
$\Alt(9)$&$1$&$36$&$(9,9,9,9)$\\\hline
\end{tabular}\caption{Cycle lengths for Theorem~\ref{MAIN}~(iii) with $T$ alternating}
\label{BIGtablea}
\end{center}
\end{table}

\begin{table}[!ht]
\begin{center}
\begin{tabular}{|cccl|}\hline
Socle Factor&$\ell$& degree& Cycle Lengths \\ \hline\hline
$M_{11}$&$1$ &$11$&$(11)$, $(1,5,5)$, $(2,3,6)$, $(1,2,8)$\\
 &$1$ &$12$& $(4, 8)$, $(1, 11)$, $(2, 2, 4, 4)$, $(1, 1, 5, 5)$, $(1, 2, 3, 6)$\\
 &$2$&$121$&$(11, 55, 55)$, $(22, 33, 66)$, $(11, 22, 88)$\\
 &$2$&$144$& $(4, 8, 44, 88)$\\\hline
$M_{12}$&$1$&$12$ & $(4, 8)$, $(6, 6)$, $(2, 10)$, $(1,11)$, $(3, 3, 3, 3)$, $(2, 2, 4, 4)$ \\
 & & &$(1, 1, 5, 5)$, $(1, 2, 3, 6)$, $(1, 1, 2, 8)$\\
 &$2$&$144$&$(6, 6, 66, 66)$, $(4, 8, 44, 88)$, $(2, 10, 22, 110)$\\\hline
$M_{22}$&$1$&$22$&$(11, 11)$, $(4, 6, 12)$, $(1, 7, 14)$, $(2, 10, 10)$, $(1, 7, 7, 7)$, $(2, 4, 8, 8)$ \\\hline
$M_{23}$&$1$&$23$&$(23)$, $(1, 11, 11)$, $(2, 7, 14)$, $(3, 5, 15)$\\
 &$2$&$529$&$(23,253,253)$, $(46, 161, 322)$, $(69, 115, 345)$\\\hline
$M_{24}$&$1$&$24$ &$(12, 12)$, $(3, 21)$, $(1,23)$, $ (6, 6, 6, 6)$, $(2, 2, 10, 10)$, $(1, 1, 11, 11)$\\
 & & &$(1, 2, 7, 14)$, $(1, 3, 5, 15)$, $(2, 4, 6, 12)$\\
 &$2$&$576$&$(12, 12, 276, 276)$, $(3, 21, 69, 483)$\\\hline
\end{tabular}\caption{Cycle lengths for Theorem~\ref{MAIN}~(iii) with $T$ sporadic}
\label{BIGtableb}
\end{center}
\end{table}

\begin{table}[!ht]
\begin{center}
\begin{tabular}{|ccll|}\hline
Socle Factor&$\ell$&degree& Cycle Lengths \\ \hline\hline
$\PSL_2(8)$& $1$&$28$&$(7,7,7,7)$, $(1,9,9,9)$\\
 & $1$&$36$&$(9,9,9,9)$\\\hline
$\PSL_2(11)$&$1$&$11$&$(11)$, $(1,5,5)$, $(2,3,6)$\\
 &$2$&$121$&$(11,55,55)$, $(22,33,66)$\\\hline
$\PSL_2(16)$&$1$&$68$& $(17, 17, 17, 17)$\\\hline

$\PSL_2(19)$&$1$&$57$&$(19,19,19)$\\\hline
$\PSL_4(3)$&$1$&$130$&$(10,40,40,40)$\\\hline

$\PSU_3(3)$&$1$&$28$& $( 7,7,7,7)$, $(1,3,12,12)$, $(4,8,8,8)$\\
 &$1$&$36$&$( 6,6,12,12 )$\\\hline
$\PSU_3(5)$&$1$&$50$& $( 5, 5, 20, 20 )$\\\hline
$\PSU_4(3)$&$1$&$112$&$(28,28,28,28)$\\\hline
$\PSp_6(2)$&$1$&$28$& $(7, 7, 7, 7)$, $(4, 8, 8, 8)$, $(1, 9, 9, 9)$, $(1, 3, 12, 12)$, $(3, 5, 5, 15)$\\
 &$1$&$36$& $( 9, 9, 9, 9 )$, $(6, 6, 12, 12 )$, $(1, 5, 15, 15 )$\\\hline
$\PSp_4(3)$&$1$&$27$&$(9,9,9)$, $(3,12,12)$\\
 &$1$&$36$&$(9,9,9,9)$, $(6,6,12,12)$\\
 &$1$&$40$&$(10,10,10,10)$, $(4,12,12,12)$\\\hline
\end{tabular}\caption{Cycle lengths for Theorem~\ref{MAIN}~(iii) with $T$ classical}
\label{BIGtablec}
\end{center}
\end{table}

\begin{remark}\label{remark1}
{\rm In order to read Tables~\ref{table1},~\ref{table2},~\ref{BIGtablea},~\ref{BIGtableb} and~\ref{BIGtablec} properly we need to make a few observations. First, in order to avoid much longer tables and duplicates we have taken into account the isomorphisms $\Alt(5)\cong \PSL_2(4)\cong \PSL_2(5)$, $\PSL_2(7)\cong \PSL_3(2)$, $\Alt(6)\cong \PSL_2(9)$, $\PSL_4(2)\cong \Alt(8)$ and $\PSU_4(2)\cong \PSp_4(3)$. So, for example, the permutation representation of degree $7$ of $\PSL_2(7)$ is considered in Table~\ref{table2} via the natural permutation representation of $\PSL_3(2)$. Analogously, the permutation representation of degree $6$ of $\PSL_2(9)$ is in Table~\ref{table1} via the natural permutation representation of $\Alt(6)$. So, the reader has to take into account these six isomorphisms in order to read our tables accurately.

Second, a group $\PSL_d(q)$ in its natural representation may appear in more than one row of Table~\ref{table2}. For example $\PSL_4(3)$ appears in Lines~$1$,~$2$,~$6$,~$7$ and~$8$ (by taking $d_1=1$ and $d_2=3$) and~$21$. The permutations of $\PGL_4(3)$ having at most four cycles are obtained by considering the contribution of all of these lines. A similar remark applies to any other group. }
\end{remark}

\newpage

\newpage

\section{The action of $\Sym(m)$ on $k$-sets}\label{sec2}
The main result of this section is Proposition~\ref{main-sym}, which determines the elements of $\Sym(m)$ having at most four cycles in their action on $k$-sets; that is, on the set of $k$-element subsets of $\{1,\ldots, m\}$. So, throughout this section $H$ denotes either $\Sym(m)$ or $\Alt(m)$ in its action on $k$-sets. Replacing $k$ by $m-k$ if necessary, we may assume that $k\leq m/2$. Our arguments exploit the interplay between the action of $H$ on $k$-sets and the action of $H$ on $\{1,\ldots,m\}$. To avoid confusion, for an element $g\in H$, by ``a cycle of $g$," we mean a cycle of $g$ in its action on $k$-sets; and by ``a cycle of $g$ on $\{1,\ldots,m\}$," we mean a cycle of $g$ in its natural action on $\{1,\ldots,m\}$. Note that, in this section, we do not insist that $H$ is primitive on $k$-sets, that is, we allow $m=2k$.

Our first lemma is a simple computation. It will be used in subsequent results to show that whenever a permutation $g$ has a long cycle in its action on $\{1, \ldots, m\}$, it must have at least $5$ distinct cycles in its action on $k$-sets.

\begin{lemma}\label{bound-necklaces}
If $g\in H$ has a cycle $C$ of length $m'$ in its action on $\{1, \ldots, m\}$, then the $k$-sets that contain $k'$ elements from $C$
lie in at least $\binom{m'}{k'}/m'$ distinct cycles of $g$.
\end{lemma}

\begin{proof}
The $k$-sets that intersect $C$ in $k'$ elements form a union of cycles of $g$.  Moreover, the $k$-sets in any one of these cycles of $g$ intersect $C$ in at most $m'$ distinct $k'$-subsets.
Since we can choose $k'$ elements from $C$ in $\binom{m'}{k'}$ different ways, the number of these $g$-cycles is at least $\binom{m'}{k'}/m'$.
\end{proof}

We are now ready to show that if $m >7$, then the groups $\Alt(m)$ and $\Sym(m)$ have no elements with at most four cycles in their actions on $k$-sets, for $k \ge 3$.

\begin{lemma}\label{sym-kge3}
If $H$ contains a permutation $g$ with at most four cycles and $k \ge 3$, then $m \le 7$.
\end{lemma}

\begin{proof}
Suppose that in the disjoint cycle representation for $g$ on $\{1, \ldots, m\}$, some cycle $C$ has length $m' \ge 7$. Let $c$ denote the set of elements of $C$, so $|c|=m'$. We claim that there exists a $k$-set $x$ with $|x\cap c|\leq m'/2$. If for every $k$-set $x$ we had $|x\cap c|>m'/2$, then $\{1,\ldots,m\}\setminus c$ would have size less than $k$ and, furthermore, any $k$-set $x$ containing $\{1,\ldots,m\}\setminus c$ would intersect $c$ in a set of size at least $\lfloor m'/2\rfloor +1$. Therefore, $(\lfloor m'/2\rfloor +1)+(m-m')\leq k$. Since $k\leq m/2$, we obtain a contradiction and our claim is proved.  Note that this means that $m-m' \ge |x\setminus(x\cap c)| \ge k-m'/2$; that is, $k \le m-m'/2$.

Recall that $k \ge 3$.  If $k-(m-m') \ge 3$ then let $x$ be a $k$-set whose intersection with $c$ is as small as possible.  If $k-(m-m') <3$ then let $x$ consist of $3$ elements of $c$ and $k-3$ elements that are not in $c$ (this is possible since $k-3<m-m'$).  Let $k'=|x\cap c|$; in both cases, we see that $3 \le k' \le \lfloor m'/2 \rfloor$.
Now, for $3 \le k' \le \lfloor m'/2 \rfloor$, the binomial coefficient $\binom{m'}{k'}$ is strictly increasing with $k'$. By Lemma~\ref{bound-necklaces}, there are at least $\binom{m'}{k'}/m'$ distinct cycles of $g$ on $k$-sets whose intersection with $c$ has cardinality $k'$. We obtain \[\binom{m'}{k'}/m' \ge \binom{m'}{3}/m'=\frac{(m'-1)(m'-2)}{6}\ge \frac{6\cdot 5}{6}=5,\]
contradicting the assumption that $g$ has at most four cycles. We conclude that every cycle of $g$ on $\{1,\ldots,m\}$ has length at most $6$. It follows that $|g| \le \lcm\{2,3,4,5,6\}=60$.

In the remainder of the proof, we assume towards a contradiction that $m \ge 8$.  Now the number of 
$k$-sets is $\binom{m}{k}\ge \binom{m}{3}$. So the number of cycles on $k$-sets 
is at least $\binom{m}{3}/60$, which is greater than $4$ when  $m \ge 13$. 
For $m=8,9,10,11,12$ we can replace $60$ by the exact value for the maximum element order $n$ in $\Sym(m)$ and  
replace $3$ with $k$ for each $k$ between $3$ and $m/2$. We find that the only  possibilities with $\binom{m}{k}/n \le 4$ are  $m=8,10,12$ with $k=3$. These three remaining cases are easily eliminated in $\mathtt{magma}$.
\end{proof}

We now consider the action of symmetric and alternating groups on $2$-sets. Again, the groups must be very small to contain elements that have at most four cycles.

\begin{lemma}\label{sym-k=2}
If $k =2$ and $H$ contains a permutation $g$ that has at most four cycles, then $m \le 9$.
\end{lemma}

\begin{proof}
Towards a contradiction, suppose that  $m \ge 10$. Then since $k =2$, for a subset $x$ of $\{1, \ldots, m\}$ of cardinality $k$, there are at least $8$ elements of $\{1, \ldots, m\}$ that are not in $x$.

Suppose that there is a cycle $C$ of length $m' \ge 10$ in the disjoint cycle representation for $g$. Let $c$ be the set of elements in $C$. Then since $k=2$, it is possible to choose $x$ so that $x \cap c$ has cardinality $2$. By Lemma~\ref{bound-necklaces}, $g$ has at least $\lfloor m'/2 \rfloor$ orbits on $k$-subsets of $\{1, \ldots, m\}$ whose intersection with $c$ has cardinality $2$. Since $m'\ge 10$ we have $\lfloor m'/2 \rfloor \ge 5$, which is a contradiction.
	
If there is a cycle $C$ of length $m'$ in the disjoint cycle representation for $g$, with $6 \le m' \le 9$, then since $k=2$, it is possible to choose $x$ so that $|x \cap c|$ takes on any value between
$\max\{0,m'-8\}$ and $2$ (inclusive). Now $g$ has at least $\lfloor m'/2 \rfloor$ orbits on $k$-subsets  of $\{1, \ldots, m\}$ whose intersection with $c$ has cardinality $2$,
 and at least one orbit on $k$-subsets of $\{1, \ldots, m\}$ whose intersection with $c$ has cardinality $i$ as long as $m'-8 \le i <2$ (and $i \ge 0$). In each case, this gives a total of at least five orbits of $g$, which is a contradiction.
 
 Thus we may assume that each cycle of $g$ has length at most $5$, and hence that $|g| \le \lcm \{2,3,4,5\} =60$. Arguing as in Lemma~\ref{sym-kge3}, we have $\binom{m}{2} /60 >4$ for $m \ge 23$ so in these cases $g$ has at least five cycles. The remaining cases (where $10 \le m \le 22$) are easily eliminated using \texttt{magma}.

\end{proof}

The following proposition summarises the previous results.

\begin{proposition}\label{main-sym}
Let $H$ be $\Alt(m)$ or $\Sym(m)$ in its natural action on the $k$-subsets of
$\{1,\ldots,m\}$ with $k\leq m/2$. If $H$ contains a permutation $g$ that has at most four cycles in this action then either $k=3$ and $m=6$; or $k=2$ and $m \le 9$; or
$k=1$.
\end{proposition}

\begin{proof}
By Lemma~\ref{sym-kge3}, if $k \ge 3$, then $m \le 7$. Since $k \le m/2$, we must have $k=3$ and $m \in \{6,7\}$. The case $k=3$ and $m=7$ is eliminated with a straightforward computation. Therefore we may assume that $k \le 2$. By Lemma~\ref{sym-k=2}, if $k=2$ then $m \le 9$.
\end{proof}

\begin{remark}\label{sym6}
{\rm Notice that the action of $\Sym(6)$ on $3$-sets is imprimitive.  This group does contain permutations with at most four cycles (for example a $6$-cycle in $\Sym(6)$ has cycle lengths $2,6,6,6$, and a $5$-cycle has cycle lengths $5,5,5,5$), but there is no primitive overgroup, so these elements do not appear in our tables. 
}\end{remark}

\section{$\PGammaL_d(q)$ in its natural action}\label{sec3}

In this section we study $\PGammaL_d(q)$ acting on the points or hyperplanes of the projective space $\PGdq$ with $d \ge 2$, and we determine the cycle structure of its elements that have at most four cycles. Since these actions are permutationally isomorphic, it is sufficient to study the action on points. In order to do so, we introduce some notation (we follow~\cite[Section~$2$]{GMPSorders}).

\begin{notation}\label{not1}{\rm 
Let $s$ be a semisimple element of $\PGL_d(q)$ and let $\bar{s}$ be a semisimple 
element of $\GL_d(q)$ projecting to $s$ in $\PGL_d(q)$. The action of the matrix 
$\bar{s}$ on the $d$-dimensional vector space $V=\mathbb{F}_{q}^d$ naturally 
defines the structure of an $\mathbb{F}_{q}\langle \bar{s}\rangle$-module on $V$. 
Since $\bar{s}$ is semisimple, $V$ decomposes, by Maschke's theorem, as a direct 
sum of irreducible $\mathbb{F}_{q}\langle \bar{s}\rangle$-modules, that is, 
$V=V_1\oplus\cdots\oplus V_{l}$, with $V_i$ an irreducible $\mathbb{F}_{q}\langle 
\bar{s}\rangle$-module. Relabelling the index set $\{1,\ldots,l\}$ if necessary, 
we may assume that the first $t$ submodules $V_1,\ldots,V_t$ are pairwise non-isomorphic 
(for some $t\in\{1,\ldots,l\}$) and that for $j\in\{t+1,\ldots,l\}$, $V_j$ is isomorphic 
to some $V_i$ with $i\in\{1,\ldots,t\}$. Now, for $i\in \{1,\ldots,t\}$, let $\mathcal{W}_i
=\{W\leq V\mid W\cong V_i\}$, the set of $\mathbb{F}_{q}\langle \bar{s}\rangle$-submodules 
of $V$ isomorphic to $V_i$ and write $W_i=\sum_{W\in \mathcal{W}_i}W$. The module $W_i$ is 
usually referred to as the \emph{homogeneous} component of $V$ corresponding to the simple 
submodule $V_i$. We have $V=W_1\oplus \cdots \oplus W_t$. Set $a_i=\dim_{\mathbb{F}_{q}}(W_i)$. 
Since $V$ is completely reducible, we have $W_i=V_{i,1}\oplus\cdots\oplus V_{i,m_i}$ for some 
$m_i\geq 1$, where $V_{i,j}\cong V_i$, for each $j\in \{1,\ldots,m_i\}$. Thus we have $a_i=
d_im_i$, where $d_i=\dim_{\mathbb{F}_{q}}V_i$, and $%\displaystyle
\sum_{i=1}^{t}d_im_i=d$. For 
$i\in \{1,\ldots,t\}$, we let $x_i$ (respectively $y_{i,j}$) denote the element in $\GL(W_i)$ 
(respectively $\GL(V_{i,j})$) induced by the action of $\bar{s}$ on $W_i$ (respectively 
$V_{i,j}$). In particular, $x_i=y_{i,1}\cdots y_{i,m_i}$ and $\bar{s}=x_1\cdots x_t$.
We note further that 
\[ 
p(s)=(\underbrace{d_1,\ldots,d_1}_{m_1\textrm{ times}},\underbrace{d_2,\ldots,d_2}_{m_2\textrm{ times}},
\ldots,\underbrace{d_t,\ldots,d_t}_{m_t\textrm{ times}}) 
\]
is a partition of $n$. 

We conclude this notation by introducing some basic facts about $\cent {\GL_d(q)}{\bar{s}}$, the centralizer of $\bar{s}$ in $\GL_d(q)$, that will be used later in the proof of Theorem~\ref{thm}. Let $c \in \cent {\GL_d(q)}{\bar{s}}$. Given $i\in \{1,\ldots,t\}$ and 
$W\in \mathcal{W}_i$, we see that $W^c$ is an $\mathbb{F}_{q}\langle \bar{s}\rangle$-submodule 
of $V$ isomorphic to $W$ (because $c$ commutes with $\bar{s}$). Thus $W^c\in\mathcal{W}_i$. This shows 
that $W_i$ is $\cent{\GL_d(q)}{\bar{s}}$-invariant. It follows that 
\[ 
\cent{\GL_d(q)}{\bar{s}}=\cent{\GL(W_1)}{x_1}\times\cdots\times \cent{\GL(W_t)}{x_t} 
\] 
and every unipotent element of $\cent {\GL_d(q)}{\bar{s}}$ is of the 
form $u=u_1\cdots u_t$ with $u_i\in \cent{\GL(W_i)}{x_i}$ unipotent in $\GL(W_i)$, for each $i$. 

Now~\cite[Proposition~$2.6$]{GMPSorders} implies that $|u|\leq \max\{p^{\lceil \log_p(m_1)\rceil},\ldots,p^{\lceil \log_p(m_t)\rceil}\}$.  (Although~\cite[Proposition~$2.6$]{GMPSorders} is stated for $\PGL_d(q)$, its proof applies verbatim to $\GL_d(q)$ also.)
}
\end{notation}

We start with a preliminary number-theoretic lemma.

\begin{lemma}\label{numbertheory}
Let $q\geq 2$ and $d_1,d_2\geq 1$ be integers with $\gcd(d_1,d_2)=1$. Then there exists an integer $c$ with $\gcd(c,q-1)=1$ and
\[
\gcd(cd_1-d_2,q-1)=\left\{
\begin{array}{lll}
2&&\textrm{if }q,d_1,d_2 \textrm{ are odd},\\
1&&\textrm{otherwise}.
\end{array}
\right.
\]
\end{lemma}
\begin{proof}
%If $q=2$, we may choose $c=1$. Now suppose that $q>2$ and write $q-1=p_1^{\alpha_1}\cdots p_t^{\alpha_t}$ with $p_1,\ldots,p_t$ distinct primes and $\alpha_1,\ldots,\alpha_t\geq 1$.
%
%Observe that for each $p_i>2$, either $d_2+1$ or $d_2-1$ is coprime to $p_i$. In particular, for $p_i>2$, choose $\varepsilon_i\in \{1,-1\}$ with $d_2+\varepsilon_i$ coprime to $p_i$.
%
%Now using the Chinese Reminder Theorem choose $c\in \mathbb{Z}$ with 
%\begin{align*}
%&cd_1\equiv d_2+\varepsilon_i\mod p_i &&\textrm{if  }p_i \textrm{ is coprime to }d_1, \textrm{or}\\
%&c\equiv 1\mod p_i&&\textrm{if } p_i \textrm{ divides }d_1,
%\end{align*}
% for each $p_i>2$ and with
%\begin{align*}
%&cd_1\equiv d_2+2\mod 4 &&\textrm{if  }d_1 \textrm{ and }d_2 \textrm{ are odd}, \textrm{or}\\
%&c\equiv 1\mod p_i&&\textrm{if } d_1 \textrm{ or }d_2 \textrm{ are even},
%\end{align*}
% when $p_i=2$. Since $\gcd(d_1,d_2)=1$, it follows that $c$ and $cd_1-d_2$ are both coprime to $p_i$ when $p_i>2$. Moreover, $c$ is odd and $cd_1-d_2$ is either congruent to $2\mod 4$ (when $q,d_1,d_2$ are odd) or odd (when $q$, $d_1$ or $d_2$ is even). \\
% ***********************New version****\\
 If $q=2$, we may choose $c=1$. Now suppose that $q>2$ and write $q-1=p_1^{\alpha_1}\cdots p_t^{\alpha_t}$ with $p_1,\ldots,p_t$ distinct primes and $\alpha_1,\ldots,\alpha_t\geq 1$.

Observe that for each $p_i>2$, either $d_2+1$ or $d_2-1$ is coprime to $p_i$. In particular, for $p_i>2$, choose $\varepsilon_i\in \{1,-1\}$ with $d_2+\varepsilon_i$ coprime to $p_i$.

Now using the Chinese Reminder Theorem, we choose $c\in \mathbb{Z}$ with 
\begin{align*}
&cd_1\equiv d_2+\varepsilon_i\mod p_i &&\textrm{if  }p_i \textrm{ is coprime to }d_1, \textrm{or}\\
&c\equiv 1\mod p_i&&\textrm{if } p_i \textrm{ divides }d_1,
\end{align*}
 for each $p_i>2$ and with
\begin{align*}
&cd_1\equiv d_2+2\mod 4 &&\textrm{if  }d_1 \textrm{ and }d_2 \textrm{ are odd}, \textrm{or}\\
&c\equiv 1\mod p_i &&\textrm{if } d_1 \textrm{ or }d_2 \textrm{ is even}
\end{align*}
if $p_i=2$.
 Since $\gcd(d_1,d_2)=1$, it follows that $c$ and $cd_1-d_2$ are both coprime to $p_i$ when $p_i>2$. In particular, the result follows for $q$ even. If $q$ is odd then there there is a $p_i$ equal to $2$ and so $c$ is odd so $\gcd(c,q-1)=1$. Moreover $cd_1-d_2$ is either congruent to $2\mod 4$ (when $d_1,d_2$ are odd, in which case $\gcd(cd_1-d_2,q-1)=2$) or odd (when $d_1$ or $d_2$ is even, in which case $\gcd(cd_1-d_2,q-1)=1$).
\end{proof}

\begin{theorem}\label{thm}
If $\PGammaL_d(q)$ contains a permutation $g$ with at most four cycles, then $d$, $q$ and the cycle lengths of $g$ are as in Table~\ref{table2} with $\ell=1$.
\end{theorem}
\begin{proof}
Let $V=\mathbb{F}_q^d$ be the $d$-dimensional vector space of row vectors over the finite field $\mathbb{F}_q$ of size $q$ and let $e_1,\ldots,e_d$ be the canonical basis of $V$. Write $q=p^f$ for some prime $p$ and some integer $f\geq 1$. For a subspace $W$ of $V$, we denote by $P(W)$ the set of $1$-dimensional subspaces of $V$ contained in $W$. Moreover we use Notation~\ref{not1}.

We split the proof into two cases depending on whether $g\in \PGL_d(q)$, or $g\in \PGammaL_d(q)\setminus\PGL_d(q)$.

Suppose that $g\in \PGL_d(q)$ has at most four cycles and observe that $|g|\geq (q^d-1)/(4(q-1))$. Moreover, we start by dealing with the case that $g$ is a semisimple element. By Maschke's theorem, as in Notation~\ref{not1}, $V=V_1\oplus\cdots\oplus V_t$, where each $V_i$ is an irreducible $\mathbb{F}_q\langle g\rangle$-module. If $t\geq 3$, then $g$ has at least $5$ orbits on $P(V)$. Indeed, $g$ has at least one orbit on each of $P(V_1)$, $P(V_2)$, $P(V_3)$, $P(V_1\oplus V_2)\setminus (P(V_1)\cup P(V_2))$ and $P(V_1\oplus V_3)\setminus (P(V_1)\cup P(V_3))$. Thus $t\leq 2$. 

If $t=1$, then $g$ acts irreducibly on $V$ and so by Schur's lemma $g$ lies in a Singer cycle of $\PGL_d(q)$ of order $(q^d-1)/(q-1)$. As a Singer cycle acts regularly on $P(V)$, we have $|g|=(q^d-1)/(i(q-1))$ for $i=1,2,3,4$ such that $i\mid (q^d-1)/(q-1)$, and $g$ has exactly $i$ cycles in its action on $P(V)$. Note that $2$ divides $(q^d-1)/(q-1)$ if and only if $q$ is odd and $d$ is even. Similarly, $3$ divides $(q^d-1)/(q-1)$ if and only if $3\mid d$ and $q\equiv 1\pmod 3$, or $2\mid d$ and $q\equiv 2\pmod 3$. Finally, $4$ divides $(q^d-1)/(q-1)$ only if $4\mid d$ and $q\equiv 1\pmod 4$, or $2\mid d$ and $q\equiv 3\pmod 4$. This gives the first six lines of Table~\ref{table2}. 

Now suppose that $t=2$. Write $d_i=\dim_{\mathbb{F}_q} V_i$ for $i=1,2$. Now a lift of $g$ in $\GL_d(q)$ can be written as $g_1g_2=g_2g_1$, with $g_i\in \GL_{d_i}(q)$ fixing pointwise $P(V_j)$, for $i \ne j$. 
We claim that $g_i$ has only one orbit on $P(V_i)$ for $i=1,2$. For otherwise, if $g_1$ had two orbits on $P(V_1)$ say, then $g$ would have two orbits on $P(V_1)$ and at least two orbits on $P(V)\setminus (P(V_1)\cup P(V_2))$. Since $g$ also has at least one orbit on $P(V_2)$, $g$ would have at least five orbits on $P(V)$. Therefore we argue as in the previous paragraph that, since $g_i$ has only one cycle on $P(V_i)$, it is a Singer cycle in its action on $P(V_i)$ and so has order $a_i=(q^{d_i}-1)/(q-1)$ on $V_i$.  In particular, we have $g_i=h_i^{j_i}$ for some Singer cycle $h_i\in\GL_{d_i}(q)$ and some divisor $j_i$ of $q-1$ such that $\gcd(j_i,a_i)=1$. 

As $g$ has at most two cycles on $P(V)\setminus (P(V_1)\cup P(V_2))$, for any vector $w_2 \in W_2\setminus \{0\}$, at least half of the vectors of the form $w_1\oplus w_2$ with $w_1 \in W_1\setminus \{0\}$ must appear in each of these cycles.  A vector of the form $w_1'\oplus w_2$ with $w_1' \in W_1\setminus \{0\}$ can only occur as $g^{a_2j}(w_1\oplus w_2)$ for some $j$, since $g_2$ acts regularly on $W_2 \setminus\{0\}$ with order $a_2$.  Thus the element $g^{a_2}=g_1^{a_2}$ has at most two cycles on $P(V_1)$ and hence $\gcd(a_1,a_2)\leq 2$. Moreover, since $\gcd(a_1,a_2)=(q^{\gcd(d_1,d_2)}-1)/(q-1)$ (which cannot be equal to $2$), we have $\gcd(d_1,d_2)=1$.

Since $g$ has at most four cycles, we see that $g$ has at most two cycles on $P(V)\setminus (P(V_1)\cup P(V_2))$, which is a set of
\[\frac{q^d-1}{q-1}-\frac{q^{d_1}-1}{q-1}-\frac{q^{d_2}-1}{q-1} = \frac{(q^{d_1}-1)(q^{d_2}-1)}{q-1}=a_1a_2(q-1)\]
 $1$-subspaces. 
% If $g$ has a single cycle on this set, it has 3 cycles altogether and is as described in line 7 of Table~\ref{table2}. 

We show that either
\begin{itemize}
\item $g$ has exactly one cycle on this set and $qd_1d_2$ is even (giving rise to line 7 of Table~\ref{table2}), or
\item $g$ has exactly two cycles of equal length (namely $a_1a_2(q-1)/2$) on this set and $qd_1d_2$ is odd (giving rise to line 8 of Table~\ref{table2}).
\end{itemize}
To prove this claim, arbitrarily choose $w_1, w_1'\in V_1\setminus\{0\}$ and $w_2, w_2'\in V_2\setminus\{0\}$.  Since $g_i$ is a Singer cycle on $V_i$ and fixes $V_{3-i}$ pointwise, there exist $j_1,j_2$ such that $(w_1\oplus w_2)g_1^{j_1}g_2^{j_2}=w_1'\oplus w_2'$.  Since $g_1^{j_1}g_2^{j_2}$ commutes with $g$, it is straightforward to see that the orbits of $w_1\oplus w_2$ and of $w_1'\oplus w_2'$ under $g$ must have the same length. In particular, if $g$ has two cycles on this set, then $a_1a_2(q-1)$ is even, and hence $q$ is odd. Moreover, 
\[
|g|=\left\{
\begin{array}{lcl}
a_1a_2(q-1)&&\textrm{if }g \textrm{ has three cycles on }P(V),\\
a_1a_2(q-1)/2&&\textrm{if }g \textrm{ has four cycles on }P(V).\\
\end{array}
\right.
\]

%Since, for each $i$, $g_i$ is transitive on $P(V_i)$, we have $g_i=h_i^{j_i}$ for some Singer cycle $h_i\in\GL_{d_i}(q)$ and some divisor $j_i$ of $q-1$ such that $\gcd(j_i,a_i)=1$.  
%Then $$|g|=\lcm\{|g_1|,|g_2|\}=\lcm\left\{\frac{q^{d_1}-1}{j_1},\frac{q^{d_2}-1}{j_2}\right\},$$ and since $g$ has one or two cycles on $P(V)\setminus (P(V_1)\cup P(V_2))$, it follows that $\gcd(d_1,d_2)=1$.  Now from the cycle lengths in $g$, it is clear that $|g|=a_1a_2(q-1)/2$, so $|g^{a_1a_2}|=(q-1)/2$.  
Since $h_i$ is a Singer cycle in $\GL_{d_i}(q)$, we see that $h_i^{a_i}$ corresponds to multiplication by some scalar $\lambda_i\in \mathbb F_q$, with $|\lambda_i|=q-1$.  If $\lambda_1=\lambda$, then $\lambda_2=\lambda^c$ for some $c$ such that $\gcd(c,q-1)=1$.  Now,  for any nonzero $v_1\in V_1$ and any nonzero $v_2 \in V_2$, we have
\begin{equation}\label{dumb}
g^{a_1a_2}: \langle v_1 \oplus v_2\rangle \to \langle \lambda^{a_2j_1}v_1\oplus  \lambda^{ca_1j_2}v_2\rangle=\langle v_1\oplus \lambda^{ca_1j_2-a_2j_1}v_2\rangle.
\end{equation}
This shows that $|g^{a_1a_2}|=|\lambda^{ca_1j_2-a_2j_1}|$. Therefore $\gcd(ca_1j_2-a_2j_1,q-1)=1$ when $g$ has three cycles and $\gcd(ca_1j_2-a_2j_1,q-1)=2$ when $g$ has four cycles. Suppose first that $g$ has three cycles. Then either $q$ is even (and hence $qd_1d_2$ is even), or $q$ is odd and $ca_1j_2-a_2j_1$ is odd. In the latter case, since $\gcd(j_1,q-1)=\gcd(j_2,q-1)=\gcd(c,q-1)=1$, we find that $c,j_1,j_2$ are odd and thus $a_1$ and $a_2$ have opposite parity.  Notice that the parity of $a_i$ is the same as the parity of $d_i$ since $q$ is odd and $a_i=q^{d_i-1}+\cdots+1$ (which has $d_i$  summands). This shows that $qd_1d_2$ is even in both cases.

Suppose that $g$ has four cycles.  Then $\gcd(ca_1j_2-a_2j_2,q-1)=2$ and hence $q$ is odd. Moreover, since $\gcd(j_1,q-1)=\gcd(j_2,q-1)=\gcd(c,q-1)=1$, we find that $c,j_1,j_2$ are odd and thus $a_1$ and $a_2$ have the same parity. Therefore $d_1$ and $d_2$ have also the same parity. Since $\gcd(d_1,d_2)=1$, this shows that $qd_1d_2$ is odd. Our claim is now proved.

We now show that when $qd_1d_2$ is odd, there do exist permutations with exactly four cycles on $P(V)$. Replacing $h_2$ by a suitable power, we may assume that $h_2^{a_2}=\lambda=h_1^{a_1}$. From Lemma~\ref{numbertheory}, there exists $c\in \mathbb{Z}$ with $\gcd(c,q-1)=1$ and $\gcd (cd_1-d_2,q-1)=2$. Take $g=h_1h_2^c$. The computation in~\eqref{dumb} shows that $g$ has one orbit on each of $P(V_1)$ and  $P(V_2)$, and two orbits on $P(V)\setminus(P(V_1)\cup P(V_2))$.

Similarly when $qd_1d_2$ is even, there do exist permutations with exactly three cycles on $P(V)$. Replacing $h_2$ by a suitable power, we may assume that $h_2^{a_2}=\lambda=h_1^{a_1}$. From Lemma~\ref{numbertheory}, there exists $c\in \mathbb{Z}$ with $\gcd(c,q-1)=1$ and $\gcd (cd_1-d_2,q-1)=1$. Take $g=h_1h_2^c$. The computation in~\eqref{dumb} shows that $g$ has one orbit on each of $P(V_1)$ and  $P(V_2)$, and one orbit on $P(V)\setminus(P(V_1)\cup P(V_2))$.

Next assume that $g$ is not semisimple. Write $g=g_{ss}g_u=g_ug_{ss}$, with $g_{ss}$ semisimple and $g_u$ unipotent. We use Notation~\ref{not1} for the element $g_{ss}$ and for the partition $p(g_{ss})$ associated to $g_{ss}$. 

Assume that $t=1$ so that 
\[p(g_{ss})=(\underbrace{d_1,\ldots,d_1}_{m_1\,\textrm{times}})\]
 for some $d_1,m_1\geq 1$; we note that $d=d_1m_1$. Observe that since $g$ is not semisimple $m_1\geq 2$. Now, from Notation~\ref{not1}, $|g| \le p^{\lceil\log_p(m_1)\rceil}(q^{d_1}-1)/(q-1)$. A computation shows that if $d \ge 3$, then this upper bound on $|g|$ is greater than or equal to $(q^d-1)/(4(q-1))$ only if $(q,d_1,m_1,d)=(2,1,3,3)$, $(2,1,4,4)$, $(2,1,5,5)$, $(2,2,2,4)$, $(3,2,2,4)$. Analysing individually each of these examples, we find that only $\PGL_3(2)$, $\PGL_4(2)$ and $\PGL_4(3)$ contain an element with at most four cycles that is not semisimple (and the cycle structures are $(1,2,4)$, $(3,6,6)$ and $(4,12,12,12)$ respectively). Thus we obtain lines 18, 21 and 22 of Table~\ref{table2}. Now assume that $d=2$. In this case, as $g$ is not semisimple, we must have $g_{ss}=1$ and $g=g_u$, that is, $g$ is unipotent. It is clear that $g$ has at least $1+q/p$ cycles on $P(V)$ and hence either $q$ is prime, or $q\in \{4,9\}$ (and since $p \mid |g|$, the element $g$ has cycle structure $(1,q)$, $(1,2,2)$ or $(1,3,3,3)$ respectively). Thus we obtain lines 9, 10 and 12 of Table~\ref{table2}.

Next assume that $t\geq 2$. Moreover, assume also that $m_1,m_2\geq 2$. Again from Notation~\ref{not1} and from the fact that the maximal element order of $\GL_{d-d_1m_1-d_2m_2}(q)$ is $q^{d-d_1m_1-d_2m_2}-1$ (see for example~\cite[Corollary~$2.7$]{GMPSorders}), we have \[|g|< (q^{d-d_1m_1-d_2m_2}-1)\lcm(q^{d_1}-1,q^{d_2}-1)\max\{p^{\lceil \log_p(m_1)\rceil},p^{\lceil \log_p(m_2)\rceil}\}.\]
Another computation shows that this upper bound on $|g|$ implies that $|g|<(q^d-1)/(4(q-1))$, and hence $g$ has more than four cycles. 

It remains to consider the case 
\[p(g_{ss})=(\underbrace{d_1,\ldots,d_1}_{m_1\,\textrm{times}},d_2,d_3,\ldots,d_t)\]
with $m_1\geq 1$, $m_2=\cdots =m_t=1$ and $t\geq 2$. Observe that since $g$ is not semisimple, we have $m_1\geq 2$. Now 
\begin{equation}\label{EQ}
 |g|< (q^{d-d_1m_1-d_2}-1)\lcm(q^{d_1}-1,q^{d_2}-1)p^{\lceil \log_p(m_1)\rceil}
\end{equation} 
 and hence another computation shows that $|g| \geq (q^d-1)/(4(q-1))$ is possible only if $d_1\leq 2$ and $m_1\leq 3$. And further, if $m_1=3$, then $(q,d_1)=(2,1)$. First suppose that $m_1=3$. Then since $(q,d_1)=(2,1)$, the element $g$ in its Jordan decomposition has a $3\times 3$ unipotent block 
\[\left(
\begin{array}{ccc}
1&1&0\\
0&1&1\\
0&0&1\\
\end{array}
\right),
\]
%has a $J_3$ Jordan block, % of the form
%\[
%\left(
%\begin{array}{ccc}
%1&1&0\\
%0&1&1\\
%0&0&1
%\end{array}
%\right)
%\]
which has three orbits on $\langle e_1,e_2,e_3\rangle$.
As $t\geq 2$, we have $d\geq m_1+1 \ge 4$ and it is now clear that $g$ has more than four orbits on $P(V)$. Thus we may assume that $m_1=2$. In this case, the Jordan form of $g$ is
\[
\left(
\begin{array}{ccc}
A&I&0\\
0&A&0\\
0&0&B
\end{array}
\right)
\]
where $A\in \GL_{d_1}(q)$, $B \in \GL_{d-2d_1}(q)$ is semisimple and $I \in \GL_{d_1}(q)$ is the identity matrix. It is clear that $e_1$, $e_{d_1+1}$, $e_{2d_1+1}$, $e_{d_1+1}+e_{2d_1+1}$, and $e_1+e_{2d_1+1}$ are all in different cycles of $g$, since $g$ fixes setwise the subspaces $\langle e_{d_1+1}, \ldots, e_{2d_1}\rangle$, $\langle e_{1}, \ldots e_{2d_1}\rangle$, $\langle e_{2d_1},\ldots,e_d\rangle$.  (Note that we apply matrices from the right.)  Thus $g$ has at least 5 cycles, a contradiction.
This concludes the case $g\in \PGL_d(q)$.

\smallskip

Finally suppose that $g\in \PGammaL_d(q)\setminus\PGL_d(q)$. In particular, $f\geq 2$ and $g=x\psi$ for some $x\in \PGL_d(q)$ and for some non-trivial field automorphism $\psi$ of order $e>1$. Let $\overline{\mathbb{F}}_q$ be the algebraic closure of the finite field $\mathbb{F}_p$. From Lang's theorem~\cite[Theorem~2.1]{GLS}, there exists $a$ in the algebraic group $\PGL_d(\overline{\mathbb{F}}_q)$ with $x=aa^{-\psi}$. Observe that $(x\psi)^e=xx^\psi\cdots x^{\psi^{e-2}}x^{\psi^{e-1}}$. Write $z=a^{-1}(x\psi)^ea$. Now,
\begin{eqnarray*}
z^\psi&=&a^{-\psi}(x^\psi x^{\psi^2}\cdots x^{\psi^{e-1}}x^{\psi^{e}})a^\psi=
a^{-\psi}(x^\psi x^{\psi^2}\cdots x^{\psi^{e-1}}x)a^\psi\\
&=&(a^{-\psi}x^{-1})(xx^\psi\cdots x^{\psi^{e-2}}x^{\psi^{e-1}})(xa^\psi)=a^{-1}(xx^\psi\cdots x^{\psi^{e-2}}x^{\psi^{e-1}})a\\
&=&a^{-1}(x\psi)^ea=z
\end{eqnarray*}
and so $z$ is invariant under the field automorphism $\psi$ of order $e$. Thus $z\in \PGL_d(q^{1/e})$. From~\cite[Corollary~$2.7$]{GMPSorders} for example, we know that the maximal element order of $\PGL_d(q^{1/e})$ is $(q^{d/e}-1)/(q^{1/e}-1)$. 
Observe further that since $g$ has at most four cycles, we have $|g|\geq \lceil(q^d-1)/(4(q-1))\rceil$. Therefore
\[\left\lceil\frac{q^d-1}{4(q-1)}\right\rceil\leq |g|= e|g^e|=e|z|\leq e \left( \frac{q^{d/e}-1}{q^{1/e}-1}\right).\]
It follows by a direct computation that this inequality is satisfied only for
\begin{eqnarray*}
(p,f,e,d)&\in &\{ (2, 2, 2, 2), ( 3, 2, 2, 2), (5, 2, 2, 2),
(7, 2, 2, 2),
(2, 3, 3, 2),
(3, 3, 3, 2),\\
&&(2, 4, 2, 2),
(2, 4, 4, 2),
(2, 5, 5, 2),
(2, 6, 2, 2),
(2, 6, 6, 2),
(2, 2, 2, 3),\\
&&(3, 2, 2, 3),
(2, 3, 3, 3),
(2, 2, 2, 4)\}.
\end{eqnarray*}
For each of these quadruples, we can check with a computer whether $\PGammaL_d(q)$ contains a permutation with at most four cycles and we obtain only the cases listed in Table~\ref{table2}.
\end{proof}

\section{Primitive groups: almost simple and product action type}\label{notation}
In this section we assume that $G=\Sym(m)\wr \Sym(\ell)$ in its product action (of degree $n={m\choose k}^\ell$) on a Cartesian product of $\ell$  copies of the set of $k$-subsets of $\{1,\ldots,m\}$ (where $1 \le k <m/2$), or that $G=\PGammaL_d(q)\wr\Sym(\ell)$ in its product action (of degree $n=((q^d-1)/(q-1))^\ell$) on a Cartesian product of $\ell$  copies of the set of  points of the projective space $\PGdq$. We do not require that $\ell\geq 2$ and so, in particular, we deal simultaneously with primitive groups of almost simple and product action type. Write $\Delta$ for the family of $k$-subsets of $\{1,\ldots,m\}$ or for the point set of $\PGdq$ (depending on which group we are considering) and $\Omega=\Delta^\ell$.

We let $\soc(G)$ denote the socle of $G$; here we have $\soc(G)=\Alt(m)^\ell$ or $\soc(G)=\PSL_d(q)^\ell$. Moreover, in what follows we write $H=\Sym(m)$ or $H=\PGammaL_d(q)$ (again depending on which group we are considering) and we write the elements $g \in G$ as $(h_1,\ldots,h_\ell)\sigma$, with $h_1,\ldots,h_\ell\in H$ and $\sigma\in\Sym(\ell)$. We reserve the letter $T$ to denote the socle of $H$ and $r=|\Delta|$ to denote the degree of the permutation group $H$ (so, $r={m\choose k}$ or $r=(q^d-1)/(q-1)$). Since we are only interested in primitive groups of almost simple and product action type, we assume that $m\geq 5$, $d\geq 2$, and $q\geq 4$.  Since $r \neq 1$, this forces $r \ge 5$.

We start with a basic lemma that will be used repeatedly in the sequel.
\begin{lemma}\label{tech}
Let $\Delta_i$ be a finite set and let $h_i\in \Sym(\Delta_i)$ be a cycle of length $|\Delta_i|$, for $i=1,2$. Then $(h_1,h_2)$ has $\gcd(|\Delta_1|,|\Delta_2|)$ cycles in its action on $\Delta_1\times \Delta_2$ and moreover each cycle has the same length.
\end{lemma}
\begin{proof}
See~\cite[Lemma~$3$, page~$92$]{Neu} or~\cite[$3.2.19$]{DM}.
\end{proof}

The following two results are the main tools used in this section.
\begin{proposition}\label{prop1:PA}Assume that $G$ contains a permutation $g=y\sigma$ having at most four cycles, with $y\in H^\ell$, $\sigma\in\Sym(\ell)$ and $\sigma\neq 1$. Then $\sigma$ is a transposition, and either $r\in \{5,6,7,8\}$ and $\ell=2$, or $r\in\{5,8\}$ and $\ell=3$.
\end{proposition}
\begin{proof}
We have $g=(h_1,\ldots,h_\ell)\sigma$ with $h_1,\ldots,h_\ell\in H$. Let $t$ be the length of the longest cycle of $\sigma$ and let $c_g$ be the number of cycles of $g$. Relabelling the index set $\{1,\ldots,\ell\}$ if necessary, we may assume that $(1,\ldots,t)$ is a cycle of $\sigma$. Consider $x=g^t$ and let $c_x$ be the number of cycles of $x$. Clearly, we have $c_g\geq \lceil c_x/t\rceil$ and so
\begin{equation}\label{eqeqeq1}\lceil c_x/t\rceil \leq 4.
\end{equation} Now we estimate $c_x$. We have
\[x=(h_1\cdots h_t,h_2\cdots h_th_1,\ldots, h_th_1\cdots h_{t-1},h_{t+1}',\ldots,h_\ell')\tau\]
with $\tau=\sigma^t$ and with $h_{t+1}',\ldots,h_\ell'\in H$. We note that the elements \[h_1h_2\cdots h_t,\, h_2\cdots h_th_1,\,\ldots,\, h_th_1\cdots h_{t-1}\] are pairwise conjugate permutations of $H$. In particular, replacing $x$ by a suitable conjugate in $G$, we may assume that
\[x=(\underbrace{h,\ldots,h}_{t\,\textrm{times}},h_{t+1}',\ldots,h_{\ell}')\tau\]
with $h\in H$.
Note that $\tau$ fixes point-wise the indices $\{1,\ldots,t\}$.
Denote by $d_i$ the number of cycles of $h$ of length $i$, for $i\in \{1,\ldots,r\}$. Let $(\alpha_{1,1},\alpha_{2,1},\ldots,\alpha_{i,1}),\ldots, (\alpha_{1,d_i},\alpha_{2,d_i},\ldots,\alpha_{i,d_i})$ be the $d_i$ cycles of length $i$ of $h$. Write \[X_i=\bigcup_{j=1}^{d_i}\{\alpha_{1,j},\alpha_{2,j},\ldots,\alpha_{i,j}\}\] and fix an element $\overline{\alpha}$ of $X_i$. It is easy to check that for each choice $\beta_1,\ldots,\beta_{t-1}\in X_i$, the point \[(\overline{\alpha},\beta_1,\ldots,\beta_{t-1},\underbrace{\overline{\alpha},\ldots,\overline{\alpha}}_{(\ell-t)\,\textrm{times}})\] 
is in a distinct cycle of $x$. Since $|X_i|=d_ii$, we have
\begin{equation}\label{eqeqeq2}
c_x\geq \sum_{i=1}^r(d_ii)^{t-1}.
\end{equation}
Recall that $r=\sum_{i}d_ii$. For $t=2$,~\eqref{eqeqeq1} and~\eqref{eqeqeq2} imply that $r\leq 8$. For $t\geq 6$, we have $a^{t-1}>4t$ (for $a\geq 2$) and so~\eqref{eqeqeq1} and~\eqref{eqeqeq2} are not satisfied. Assume that $t=3$. Now $a^{t-1}>4t$ for $a\geq 4$. In particular,~\eqref{eqeqeq1} and~\eqref{eqeqeq2} are satisfied only if $d_1\leq 3$, $d_2\leq 1$, $d_3\leq 1$ and $d_i=0$ for $i\geq 4$. Thus $r\leq 8$. Now, with a direct inspection for every possible value of $r\in \{5,\ldots,8\}$, we see that~\eqref{eqeqeq1} and~\eqref{eqeqeq2} are never satisfied. A similar detailed analysis for $t=4$ and $t=5$ yields no solutions to~\eqref{eqeqeq1} and~\eqref{eqeqeq2}.

Summing up, we have $t=2$ (so that $\sigma$ is an involution) and $r\leq 8$. If $\sigma$ has at least two cycles of length 2, then the same argument as above produces at least $(\sum d_ii )(\sum d_i' i)=r^2$ distinct cycles of $x$.  This is a contradiction since $c_x \le 2c_g \le 8$ and $r^2 \ge 25$.  Thus $\sigma$ is a transposition.

Since $r$ is bounded by $8$, we see from~\cite[Theorem~$1.3$]{GMPSorders} that $\ell$ is also bounded above by a constant. In fact, using~\cite[Table~$6$]{GMPSorders}, we see that $\ell\leq 3$ if $r=5$ or $6$, and $\ell\leq 4$ if $r=7$ or $r=8$. Now the proof follows by a computer calculation. In particular, we have $\ell\leq 3$ and, for $r\in \{6,7\}$, we have $\ell=2$.
\end{proof}

The following proposition deals with the case where our permutation $g$ with at most four cycles lies in the base group $H^\ell$ of $G$.

\begin{proposition}\label{prop2:PA}
Assume that $H^\ell$ (for $\ell \geq 2$) contains a permutation $g=(h_1,\ldots,h_\ell)$ having at most four cycles. Then, relabelling the index set $\{1,\ldots,\ell\}$ if necessary, one of the following holds:
\begin{enumerate}
\item[(i)] $\ell=2$, $h_1$ and $h_2$ both have two cycles on $\Delta$ of lengths $t_1,r-t_1$ and $t_2,r-t_2$, and each of $t_1,r-t_1$ is relatively prime to each of $t_2,r-t_2$;
\item[(ii)] $\ell=2$, $h_1$ is an $r$-cycle, $h_2$ has two cycles on $\Delta$ of lengths $t_2,r-t_2$; moreover $\gcd(r,t_2)\leq 2$;
\item[(iii)]$\ell=2$, $h_1$ is an $r$-cycle, $h_2$ has three cycles on $\Delta$ of lengths $t_2,t_2',r-(t_2+t_2')$, moreover $\gcd(r,t_2)+\gcd(r,t_2')+\gcd(r,t_2+t_2')\leq 4$;
\item[(iv)]$\ell=3$, $h_1$ is an $r$-cycle, $h_2$ has two cycles on $\Delta$ of lengths $t_2,r-t_2$, and also $h_3$ has two cycles on $\Delta$ of lengths $t_3,r-t_3$; moreover $r,t_2,r-t_2,t_3,r-t_3$ are pairwise relatively prime;
\item[(v)]$\ell=2$, $h_1$ is an $r$-cycle, $h_2$ has four cycles on $\Delta$ of lengths $t_2,t_2',t_2'',r-(t_2+t_2'+t_2'')$; moreover $\gcd(r,t_2)=\gcd(r,t_2')=\gcd(r,t_2'')=\gcd(r,t_2+t_2'+t_2'')=1$.
\end{enumerate}
\end{proposition}
\begin{proof}
Let $c_g$ be the number of cycles of $g$ on $\Omega$ and let $c_i$ be the number of cycles of $h_i$ on $\Delta$. Replacing $g$ by a suitable conjugate in $G$, we may assume that $c_1\leq c_2\leq\cdots \leq c_\ell$. Clearly,
\begin{equation}\label{eq3}
c_g\geq \prod_{i=1}^{\ell}c_i.
\end{equation}
Assume that $c_1\geq 2$. As $c_g\leq 4$, from~\eqref{eq3}, we deduce that $\ell=2$ and $c_1=c_2=2$.
Write
\begin{eqnarray*}
h_1&=&(\alpha_1,\ldots,\alpha_{t_1})(\alpha_{t_1+1},\ldots,\alpha_{r})\\
h_2&=&(\beta_1,\ldots,\beta_{t_2})(\beta_{t_2+1},\ldots,\beta_{r}).
\end{eqnarray*}
If $\gcd(t_1,t_2)\geq 2$, then $g=(h_1,h_2)$ has at least two orbits on $\{\alpha_1,\ldots,\alpha_{t_1}\}\times \{\beta_1,\ldots,\beta_{t_2}\}$ from Lemma~\ref{tech}. Thus, $g$ has at least five orbits on $\Omega$ (the elements $(\alpha_{t_1+1},\beta_1)$, $(\alpha_{t_1+1},\beta_{t_2+1})$ and $(\alpha_1,\beta_{t_2+1})$ are in distinct $g$-orbits), which is a contradiction. Thus $t_1$ and $t_2$ are relatively prime. Similarly, each of $t_1,r-t_1$ is relatively prime to each of $t_2,r-t_2$. This yields part~(i).

Assume that $c_1=1$. Write $h_1=(\alpha_1,\ldots,\alpha_r)$. Suppose that $c_2=1$. From Lemma~\ref{tech}, we see that $(h_1,h_2)$ has $|\Delta|\geq 5$ orbits on $\Delta\times \Delta$, contradicting the fact that $c_g\leq 4$. Therefore $c_2\geq 2$. In particular,~\eqref{eq3} implies that either $\ell=3$ and $c_2=c_3=2$, or $\ell=2$ and $2\leq c_2\leq 4$.

Assume that $\ell=3$, $c_1=1$, $c_2=c_3=2$. Let $t_2,r-t_2$ and $t_3,r-t_3$ be the cycle lengths of $h_2$ and $h_3$. An application of Lemma~\ref{tech} to each of $(h_1,h_2)$, $(h_1,h_3)$ and $(h_2,h_3)$ implies that $r,t_2,r-t_2,t_3,r-t_3$ are relatively prime. This yields part~(iv).

Assume that $\ell=2$, $c_1=1$, $c_2=2$. Let $t_2,r-t_2$ be the cycle lengths of $h_2$. Note that $\gcd(r,t_2)=\gcd(r,r-t_2)$. An application of Lemma~\ref{tech} to $(h_1,h_2)$ implies that $\gcd(r,t_2)\leq 2$. This yields part~(ii).

The remaining cases ($c_2=3$ and $c_2=4$) follow with a similar argument giving rise to parts~(iii) and~(v).
\end{proof}

With a slight improvement of Proposition~\ref{prop2:PA} we are now able to prove Theorem~\ref{MAIN}.

\begin{proof}[Proof of Theorem~\ref{MAIN}]
From~\cite[Theorem~$1.3$]{GMPSorders} (which we restated in Section 1.1) and the discussion that follows our restatement of this result,
we have only to consider the case that $G=\Sym(m)\wr \Sym(\ell)$ in its product action (of degree $n={m\choose k}^\ell$) on a Cartesian product of $\ell$  copies of the set of $k$-subsets of $\{1,\ldots,m\}$ (with $1\leq k< m/2$), or that $G=\PGammaL_d(q)\wr\Sym(\ell)$ in its product action (of degree $n=((q^d-1)/(q-1))^\ell$) on a Cartesian product of $\ell$  copies of the point set of the projective space $\PGdq$. 

Assume that $\ell=1$. If $G=\Sym(m)$ acting on $k$-sets with $k\geq 2$, then by Proposition~\ref{main-sym} and Remark~\ref{sym6} we have $k=2$ and $m\leq 9$. For each of these cases, we list the structure of the permutations having at most four cycles in Table~\ref{BIGtablea}. If $k=1$, then we have Line~$1$ of Table~\ref{table1}. If $G=\PGammaL_d(q)$, then the proof follows from Theorem~\ref{thm} and we obtain Lines~$1$--$23$ of Table~\ref{table2}.

Now suppose that $\ell \ge 2$. Write $H=\Sym(m)$ or $H=\PGammaL_d(q)$ depending on which group we are considering. If $g\notin H^\ell$, then from Proposition~\ref{prop1:PA} we have $5\leq r\leq 8$ and $\ell=2$ or $r\in \{5,8\}$ and $\ell=3$. For each of these cases, the various possibilities for $\soc(G)$ and $(n_1,\ldots,n_N)$ (depending on whether $\soc(G)=\Alt(m)^\ell$ or $\soc(G)=\PSL_d(q)^\ell$) are reported in Table~\ref{table1} and Table~\ref{table2} (which have been obtained by computations in \texttt{magma}). For example the cycle lengths $12,12,12$ in Line~29 and the cycle lengths $16,16,16,16$ in Line~30 of Table~\ref{table2} arise only from group elements $g\notin H^\ell$. Similarly, Lines~$7$,~$8$,~$9$, and~$11$ of Table~\ref{table1} arise from group elements  $g\notin H^\ell$. So from now on we may assume that $g=(h_1,\ldots,h_\ell)\in H^\ell$.

Assume that $H=\Sym(m)$. If $k>1$ then there is no $g$ with four or fewer cycles.  This follows at once from Proposition~\ref{main-sym} and another computation with \texttt{magma}. In fact, if $k>1$, then $m\leq 9$ and also $\ell\leq 4$ from~\cite[Table~$6$]{GMPSorders}. Thus we only have a finite number of relatively small groups to check. % (and as usual, the various cases arising are in Tables~\ref{BIGtablea}, ~\ref{BIGtableb} and ~\ref{BIGtablec}). 
Assume that $k=1$. Now Proposition~\ref{prop2:PA} describes the cycles lengths of $g$ and implies that $\ell\leq 3$. If $g$ satisfies Proposition~\ref{prop2:PA} part~(i), then $(n_1,\ldots,n_N)$ is in Line~$4$ of Table~\ref{table1}. If $g$ satisfies Proposition~\ref{prop2:PA} part~(ii), then $(n_1,\ldots,n_N)$ is in Line~$2$ of Table~\ref{table1} if $\gcd(r,t_2)=1$ and in Line~$6$ of Table~\ref{table1} if $\gcd(r,t_2)=2$ (with $k_1=k_2=t_2/2$ and $k_3=k_4=(r-t_2)/2$). If $g$ satisfies Proposition~\ref{prop2:PA} part~(iii), then $(n_1,\ldots,n_N)$ is in Line~$3$ of Table~\ref{table1} if $\gcd(r,t_2)=\gcd(r,t_2')=\gcd(r,t_2+t_2')=1$ and in Line~$5$ of Table~\ref{table1} if $\gcd(r,t_2)+\gcd(r,t_2')+\gcd(d,t_2+t_2')=4$. If $g$ satisfies Proposition~\ref{prop2:PA} part~(iv), then $(n_1,\ldots,n_N)$ is in Line~$10$ of Table~\ref{table1}. Finally, if $g$ satisfies Proposition~\ref{prop2:PA} part~(v), then $(n_1,\ldots,n_N)$ is in Line~$6$ of Table~\ref{table1}.

Now assume that $H=\PGammaL_d(q)$. Again, $g=(h_1,\ldots,h_\ell)$ satisfies one of parts (i)--(v) of Proposition~\ref{prop2:PA} with $r=(q^d-1)/(q-1)$. Since $\PSL_2(4)=\Alt(5)$, we may assume that $(d,q)\neq (2,4)$. Note that we can use Table~\ref{table2} for the cycle lengths of the elements $h_1,\ldots,h_\ell$. Recall when examining the table that we may assume $r =(q^d-1)/(q-1)\ge 5$.

Assume that $g=(h_1,h_2)$ satisfies Proposition~\ref{prop2:PA} part~(i). A careful inspection of Table~\ref{table2} shows that $d=2$ and either $q$ is prime, $h_1$ and $h_2$ have cycle lengths $1,q$ and $(q+1)/2,(q+1)/2$, or $q=9$ and $h_1$ and $h_2$ have cycle lengths $2,8$ and $5,5$. In the first case $(h_1,h_2)$ has cycle lengths $(q+1)/2,(q+1)/2,q(q+1)/2,q(q+1)/2$ and this example is in Table~\ref{table2} Line~$27$. In the second case, $(h_1,h_2)$ has cycle lengths $10,10,40,40$ and this example (simply by an arithmetical coincidence) is in Table~\ref{table2} Line~$25$.

Now suppose that $g=(h_1,h_2)$ satisfies Proposition~\ref{prop2:PA} part~(ii). A careful inspection of Table~\ref{table2} shows that $d=2$ and either (a) $q$ is prime and $h_1$ and $h_2$ have cycle lengths $q+1$ and $1,q$, or (b) $q=9$ and $h_1$ and $h_2$ have cycle lengths $10$ and $2,8$. In case (a), $(h_1,h_2)$ has cycle lengths $q,(q+1)q$ and this example is in Line~$26$ of Table~\ref{table2}. In case (b), $(h_1,h_2)$ has cycle lengths $10,10,40,40$ and this example is in Line~$25$ of Table~\ref{table2} (again by coincidence).

Next, suppose that $g=(h_1,h_2)$ satisfies Proposition~\ref{prop2:PA} part~(iii). A careful inspection of Table~\ref{table2} shows that one of the following holds:
\begin{enumerate}
\item $d=d_1+d_2$, $\gcd(d_1,d_2)=1$ and $h_2$ has cycle lengths $(q^{d_1}-1)/(q-1),(q^{d_2}-1)/(q-1),(q^{d_1}-1)(q^{d_2}-1)/(q-1)$; or
\item $d=2$, $q=4$ and $h_2$ has cycle lengths $1,2,2$; or
\item $d=2$, $q=9$ and $h_2$ has cycle lengths $1,3,6$; or
\item $d=2$, $q=16$ and $h_2$ has cycle lengths $1,8,8$, or $2,3,12$, or $2,5,10$; or
\item $d=3$, $q=2$ and $h_2$ has cycle lengths $1,2,4$.
\end{enumerate}
The examples in Cases~$(2)$--$(5)$ are in Lines~$28$~$31$,~$32$ and $34$ of Table~\ref{table2}. We now consider Case~$1$. Note that $\gcd((q^d-1)/(q-1),(q^{d_i}-1)/(q-1))=1$ for $i=1,2$. Also,

\[\gcd\left(\frac{q^d-1}{q-1},\frac{(q^{d_1}-1)(q^{d_2}-1)}{q-1}\right)=\gcd(d,q-1).\]
In particular, we obtain Line~$24$ of Table~\ref{table2} if $\gcd(d,q-1)=1$ and Line~$25$ of Table~\ref{table2} if $\gcd(d,q-1)=2$.

A careful inspection of Table~\ref{table2} shows that there are no permutations $h_2,h_3 \in \PGammaL_d(q)$ satisfying all of the requirements of Proposition~\ref{prop2:PA} part~(iv).

Finally, suppose that $g=(h_1,h_2)$ satisfies Proposition~\ref{prop2:PA} part~(v). A careful inspection of Table~\ref{table2} shows that
one of the following holds:
\begin{enumerate}
\item $d=d_1+d_2$, $\gcd(d_1,d_2)=1$, $d_1,d_2$ and $q$ are odd, $h_2$ has cycle lengths $(q^{d_1}-1)/(q-1)$, $(q^{d_2}-1)/(q-1)$, $(q^{d_1}-1)(q^{d_2}-1)/(2(q-1))$, $(q^{d_1}-1)(q^{d_2}-1)/(2(q-1))$; or
\item $d=2$, $q=9$ and $h_2$ has cycle lengths $1,3,3,3$; or
\item $d=2$, $q=27$ and $h_2$ has cycle lengths $1,9,9,9$; or
\item $d=3$, $q=4$ and $h_2$ has cycle lengths $1,4,8,8$.
\end{enumerate}
The examples in Cases~$(2)$--$(4)$ are in Lines~$31$,~$33$ and~$35$ of Table~\ref{table2}. Arguing in the same way as when $g$ is as in  part~(iii) of Proposition~\ref{prop2:PA} yields Line~$25$ of Table~\ref{table2}.
\end{proof}

\end{document}